\documentclass[12pt]{amsart}
\usepackage[margin=1in]{geometry}
\usepackage{amssymb, amscd, amsmath, url}
\usepackage{lipsum}
\usepackage{etoolbox}
\usepackage[numbers]{natbib}
\usepackage{natbib}
\patchcmd{\subsection}{-.5em}{.5em}{}{} 
\patchcmd{\subsubsection}{-.5em}{.5em}{}{} 
\patchcmd{\section}{\normalfont}{\bfseries}{}{}
\newtheorem{theorem}{Theorem}[section]

\numberwithin{equation}{section}
\usepackage{listings}
\usepackage{matlab-prettifier}      
\usepackage{graphicx}
\usepackage[foot]{amsaddr}
\usepackage{float}
\theoremstyle{plain}
\newtheorem{thm}{Theorem}[section]

\newtheorem{lem}[thm]{Lemma}

\theoremstyle{definition}

\theoremstyle{remark}
\newtheorem{rem}{Remark}[section]
\numberwithin{equation}{section}

\begin{document}
\title[]{Exploring the Pseudo-modes of Schrödinger Operators with Complex Potentials: A Focus on Resolvent Norm Estimates and Spectral Stability}
	
	
	
\author{\bfseries Sameh Gana$^*$ }
\address{Department of Basic Sciences, 
Deanship of Preparatory Year and Supporting Studies, 
Imam Abdulrahman Bin Faisal University, P.O. Box 1982, Dammam, 34212, Saudi Arabia.}
\address{$^*$Corresponding author:\textnormal{sbgana@iau.edu.sa}}
\begin{abstract}
This paper aims to investigate the pseudo-modes of the one-dimensional Schrödinger operator with complex potentials, focusing on the behavior of the resolvent norm along specific curves in the complex plane and assessing the stability of the spectrum under small perturbations. The study builds upon previous work of E.B. Davies, L.S. Boulton, and N. Trefethen, specifically examining the resolvent norm of the complex harmonic oscillator along curves of the form $z_{\eta }= b\eta + c\eta ^{p} $ where $ b > 0$, $ \frac{1}{3}< p <3 $ independent of $\eta> 0$. The present work narrows the focus to the case where $p = \frac{1}{3}$. Numerical computations of pseudo-eigenvalues are performed to verify spectral instability.
\end{abstract}
\maketitle
\section{Introduction}
The study of pseudo-spectra has developed as a significant area of research within operator theory, especially in analyzing non-self-adjoint operators. The pseudo-spectrum of an operator provides critical insights into its stability under perturbations, which is crucial for many applications in quantum physics, fluid dynamics and control systems (see ~\cite{Orszag,Gana2,Gana3,Gana5,Zworskin130,Mhadhbi1,Mhadhbi2}). We refer to ~\cite{Daviesn20,Daviesn30,Daviesn40,Trefethenn100,Johannesn120} for basic ideas and general results about pseudo-spectra.\\
Let  $\mathcal{H}$ be a complex Hilbert space and $A$:
$\mathcal{H}\rightarrow\mathcal{H}$ be a closed densely defined
operator. Let $\sigma(A)$ be the spectrum of $A$. Recall that if A
is self-adjoint or, more generally, normal, we have
\begin{equation*}
\| (A-zI)^{-1}\| \leq (dist(z,\sigma (A))^{-1},
\end{equation*}
for all
$z\notin \sigma(A)$.
However, this estimation does not hold for non-self-adjoint operators, where the resolvent norm can be significantly large even when $z$ is located far from the spectrum. 
The basic idea is to study not only the points where the resolvent
is undefined, i.e, the spectrum, but also those where the resolvent norm is large.\\
Let $\varepsilon > 0$, the $\varepsilon$-pseudo-spectrum
$\sigma_{\varepsilon}(A)$ of $A$ is the set of all complex numbers
$z$ such that the resolvent norm $\| ( A-z I) ^{-1}\|$ of $A$ is greater than, or equal to $\frac{1}{\varepsilon}$:
\begin{equation*}
\sigma_{\varepsilon}(A) =\{z\in{\mathbb{C}}
: \|(A-z I)^{-1}\|\geq\dfrac{1}{\varepsilon}
\}.
\end{equation*}
We adopt the convention that $\|(A-z I)^{-1}\|=\infty$ for $z\in
\sigma(A)$.\\
Studying pseudo-spectra allows understanding the spectrum's stability under small perturbations. Indeed, the
Roch-Silbermann theorem ~\cite{Rochn140} asserts that:
$$
\sigma _{\varepsilon }( A ) =\bigcup_{\substack{\Delta A\in
\mathcal{L}(\mathcal{H}) \\ \| \Delta A\| \leq\varepsilon }} \sigma
( A +\Delta A).
$$
Thus, a number $z$ is in the interior
of the $\varepsilon$-pseudo-spectrum of $A$ if and only if it is in
the spectrum of some perturbed operator $A +\Delta A$ with $\|
\Delta A\|
\leq \varepsilon$.\\
We notice that pseudo-spectra of an operator can give more
information about the perturbation of the spectra, but this does not
mean that the study of perturbations is the main thing
pseudo-spectra are useful for. Refer to ~\cite{Trefethenn90,Katon50,horn60,horn70} for more
examples and applications.

This paper aims to investigate the pseudo-spectrum of the one-dimensional Schrödinger operator with complex potential  $$H_{c}=-\frac{d^{2}}{dx^{2}}+cx^{2},$$ where $c$ is a complex number such that  $Re(c)>0$ and $Im(c)>0 $.\\
Schrödinger operator with complex potential presents a compelling case for analysis due to its applications in open quantum systems, resonance phenomena, and PT-symmetric quantum mechanics. Unlike self-adjoint operators, where the spectrum fully captures the stability and dynamics of the system, non-self-adjoint operators can exhibit highly non-trivial pseudospectral behavior, whereby the resolvent norm becomes large far from the spectrum. This phenomenon makes the study of the pseudo-spectrum essential for understanding the sensitivity of the spectrum to perturbations and the transient behavior of the associated quantum system.\\
Recent contributions by researchers such as E.B. Davies, L.S. Boulton, and N. Trefethen ~\cite{Daviesn20,Daviesn30,Trefethenn90,TrefethenF,Boultonn10} have significantly advanced our understanding of the non-self-adjoint harmonic oscillator $H_{c}$. In particular, L.S. Boulton's work has established key results regarding the behavior of the resolvent norm for the complex harmonic oscillator. In ~\cite{Boultonn10}, the author demonstrated that the resolvent norm tends to infinity along specific curves in the complex plane, characterized by the form $z_{\eta }= b\eta + c\eta ^{p} $, where $b>0$, $ \frac{1}{3}<p <3 $, and  $c$ is a complex number with positive real and imaginary parts. This result plays a crucial role in interpreting the structure of the pseudo-spectra and suggests insights into the stability of the spectrum under small perturbations. We can deduce this conjecture from a study of Dencker, Sjöstrand
and Zworski ~\cite{Zworski}, which gives bounds on the resolvent for a semi-classical pseudo-differential operator
in a very general setting. 

In this paper, we build upon the foundational work of Boulton and others by investigating the resolvent norm's estimates along the curve defined by $z_{\eta} = b\eta + c\eta^{1/3}$. The choice of $p = \frac{1}{3}$ is motivated by the desire to explore the boundary of the previously established results and to understand the implications of this particular choice on the pseudospectral properties of the complex harmonic oscillator. By conducting numerical computations of the pseudo-spectra, we aim to verify the spectral instability predicted by the theoretical framework.\\
The structure of this paper is organized as follows: In Section 2, we provide a detailed overview of the one-dimensional Schrödinger operator named the complex harmonic oscillator, focusing on the resolvent norm, discussing its properties and the relevance of analyzing its behavior along specific curves in the complex plane of the form  $z_{\eta }= b\eta + c\eta ^{p} $ defined by L.S.Boulton ~\cite{Boultonn10}. In Section 3, we present our findings regarding the resolvent estimates in the case $p =
\frac{1}{3}$. Then, section 4 focuses on numerical computations of pseudo-eigenvalues for a discretization of the operator $H_{c} $ using spectral collocation methods and Chebfun algorithms ~\cite {TrefethenS,Driscoll,TrefethenF,TrefethenT,Hale}. Finally, Section 5 will conclude with a summary of our results and suggestions for future research directions.
\section{Some results for the complex harmonic oscillator}
We define the complex harmonic oscillator to be the operator
\begin{equation*}
H_{c}=-\frac{d^{2}}{dx^{2}}+cx^{2},
\end{equation*}
where $c$ is a complex number such that  $Re(c)>0$ and $Im(c)>0 $.\\
We assume that $H_{c}$ acts in $L^{2}(\mathbb{R})$ with Dirichlet boundary conditions.\\The domain of $H_{c}$ is
$$\mathcal{D}(H_{c})=W^{1,2}(\mathbb{R})\cap \{f\in L^{2}(\mathbb{R}):\int_{\mathbb{R}} x^{2}|f(x)|^{2}dx < \infty\}.$$
For $0\leq \alpha,\beta\leq\frac{\pi}{2}$, we define the angular
sector: $S(-\alpha,\beta )=\{z\in \mathbb{C}:-\alpha<arg(z)
<\beta\}.$
The operator $H_{c}$ is sectorial since $$Num(H_{c})\subset S(0,arg(c)),$$ 
where $Num(H_{c})$ is the
numerical range of $H_{c}$.\\
We define the closed m-sectorial quadratic form
$$h_{c}(f,g)=\int_{\mathbb{R}} f'(x) \overline{g'(x)}dx+ c\int_{\mathbb{R}}x^{2}f(x)\overline{g(x)}dx,$$
where $f$ and $g$ are in $W^{1,2}(\mathbb{R})\cap \{f\in L^{2}(\mathbb{R}):\int_{\mathbb{R}} x^{2}|f(x)|^{2}dx < \infty\}$.\\
$H_{c}$ is the m-sectorial operator
associated to $h_{c}$  via the Friedrich representation theorem (see
~\cite{Katon50} VI.$\S$2 ).\\ As proved by L.S.Boulton in
~\cite{Boultonn10}, since 
$C_{c}^{\infty}(\mathbb{R})\subset
\mathcal{S}(\mathbb{R})\subset \mathcal{D}(H_{c})$,
we conclude that
$C_{c}^{\infty}(\mathbb{R})$ and $\mathcal{S}(\mathbb{R})$ are form cores for $H_{c}$.\\
We prove that if $c\in\mathbb{R}_{+}$, $H_{c}$ is self-adjoint. However,
if $Im c\neq0$, then $H_{c}$ is no longer a normal operator. \\
Let $H_{c}^{*}$ be the adjoint of $H_{c}$,
we have:
$$H_{c}^{*}=-\frac{d^{2}}{dx^{2}}+\overline{c} x^{2},$$
and 
$\mathcal{D}(H_{c}^{*})=\mathcal{D}(H_{c}).$ \\
For any function $f\in
\mathcal{D}(H_{c}^{*})$, we have
$$H_{c}H_{c}^{*}f = f^{(4)}
-\overline{c}(x^{2}f)^{(2)}-cx^{2}f^{(2)}+|c|^{2}x^{4}f,$$
$$H_{c}^{*}H_{c}f = f^{(4)}
-c(x^{2}f)^{(2)}-\overline{c}x^{2}f^{(2)}+|c|^{2}x^{4}f.$$ Then
 $H_{c}^{*}H_{c}=H_{c}H_{c}^{*}$ if and only if $c=\overline{c}$.\\
The spectrum of $H_{c}$ is only composed of eigenvalues with
multiplicity one:
$$
\sigma (H_{c})=\{\lambda
_{n}=c^{1/2}(2n+1);n\in \mathbb{N}\}.
$$ 
It is easy to check that if $H_{\text{n}}$ is the
$n^{\text{th}}$ Hermite polynomial, then $\Psi _{n}(x)  =
c^{1/8}H_{ n}(c ^{1/4}x) e^{- c^{1/2} x^{2}/ 2}$ is an
eigenfunction of $H_{c}$ associated to the eigenvalue $\lambda
_{\text{n}}$. \\
 Figure \ref{fig:Fig2} displays the first twenty eigenvalues of $ H_{c}$. Figure \ref{fig:Fig3} shows the eigenfunctions of  $H_{c}$ associated to the eigenvalue $\lambda_{\text{1}}$, $\lambda_{\text{5}}$, $\lambda_{\text{10}}$, and $\lambda
_{\text{20}}$. 
\begin{figure}
	\caption{The first twenty eigenvalues of $ H_{c}$ for $c=1+5i$.}
	\includegraphics[width=12cm]{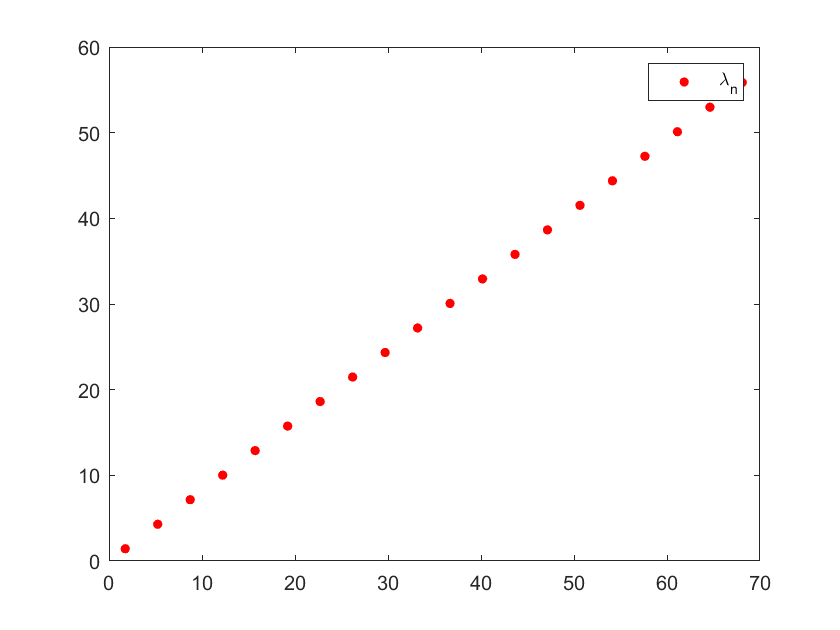}
	\label{fig:Fig2}.
\end{figure}
\begin{figure}
	\caption{The eigenfunctions of  $H_{c}$ for $c=1+5i$ associated to the eigenvalues $\lambda_{\text{1}}$, $\lambda_{\text{5}}$, $\lambda_{\text{10}}$, and $\lambda
		_{\text{20}}$.}
	\includegraphics[width=15cm]{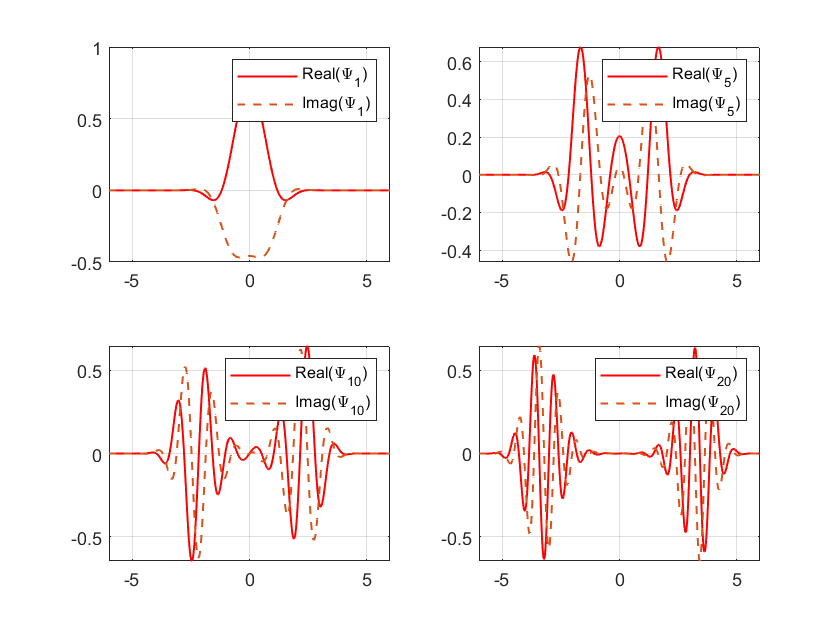}
	\label{fig:Fig3}.
\end{figure}

In ~\cite{Boultonn10}, the author proves that the
resolvent norm of $H_{c}$ tends to infinity along the curves of the
form
$z_{\eta }= b\eta + c\eta ^{p},$ where $b$ and $p$ are some constants independent of
 $\eta> 0$ such that $ b > 0$ and $\frac{1}{3} < p < 3$, i.e:
\begin{equation}\label{eq1}
	\lim_{\eta \rightarrow +\infty }\|(H_{c}-z_{\eta})^{-1}\| =\infty.
\end{equation}
On the other hand, he shows that for all $ b > 0$, there exists a constant $M_{b}$ such that
\begin{equation*}
	\lim_{\eta \rightarrow \infty }\|(H_{c}-(\eta+ib))^{-1}\| \leq M_{b}.
\end{equation*}
and
\begin{equation*}
	\lim_{\eta \rightarrow \infty }\|(H_{c}-c(\eta-ib))^{-1}\| \leq M_{b}.
\end{equation*}
 Given some numerical calculations performed by E.B. Davies in
~\cite{Daviesn20}, L.S.Boulton has conjectured that the index $ p =
\frac{1}{3}$ is the critical one. Indeed, for  $0 < p <\frac{1}{3},$ $m \in \mathbb{N}$, $0 < \delta < 1$ be fixed and $
b_{m,p}> 0 $ such that there exists $E > 0 $ verifying
$$ b_{m,p}E+ cE^{p}= \lambda_{m }.$$
Let $z_{\eta }= b_{m,p}\eta + c\eta ^{p} $ and $$\Omega_{m,p} = \{
|z_{\eta }| e^{i\theta}\in\mathbb{C} : \eta \geq E \text{ and }
arg(z_{\eta }) \leq \theta \leq arg(c\overline{z_{\eta }}/|c|).
$$ Then $$ \sigma_{\varepsilon
}(H_{c}) \subset \Omega_{m,p} \cup \bigcup_{n=0}^{m}\{ z \in
\mathbb{C} : |z-\lambda_{n}| < \delta \}. $$
This result permits us to precise the
pseudospectral shape of the complex harmonic oscillator, which is
optimal in view of (\ref{eq1}).
\\ In the next section, we study the case $p = \frac{1}{3}$ and we
prove that the resolvent norm of the operator $H_{c}$ is bounded
along $z_{\eta }$ in this case.
\section{Study of the resolvent estimates in the case $ p = \frac{1}{3}$ }
\begin{theorem}\label{thm1}
Let $H_{c}$ be the complex harmonic oscillator defined by
\begin{equation*}
H_{c}=-\frac{d^{2}}{dx^{2}}+cx^{2},
\end{equation*} where $Re(c)>0$ and $Im(c)$ $\neq 0.$%
\newline
We suppose that there exists C$_{0}>0$ such that
\begin{equation} \label{eq2} |Im(z)| \leq C_{0}(Re(z))^{\frac{1}{3}}.
\end{equation} Then:
$$
\| ( H_{c}-zI) ^{-1}\| =O( |z| ^{-\frac{1}{3}}),
$$
when $| z| \rightarrow +\infty$.
\end{theorem}
\begin{proof}
We first assume that $\ Im(c)>0$ and $Im(z)>0$.\\
Let
$$
h = \frac{1}{Re(z)}$$  and $$\mu =\frac{Im(z)}{Re(z)}.$$
When $| z| \rightarrow +\infty $, it follows from (\ref{eq2}) that $h \rightarrow 0$ and $ \mu \rightarrow 0 $.\\
 $h $ and $\mu$
are defined to be respectively a semi-classical and spectral
parameter. \\Let $x = h^{-1/2}y ,$ then the operator $P = H_{c}-z$
verifies
\begin{equation}\label{eq3}
\begin{array}{lll}
\text{ \ }P & = & - h \dfrac{d^{2}}{dy^{2}}+ ch^{-1}y^{2\text{
}}-h^{-1}-iIm(z)
\\
 & = & h^{-1}[ -h^{2}\dfrac{d^{2}}{dy^{2}} + cy^{2\text{ }}-1-i\mu
].
\end{array}
\end{equation}
Consider the pseudo-differential operator
\[
Q = -h^{2}\frac{d^{2}}{dy^{2}}+cy^{2\text{ }}-1-i\mu,
\]
of symbol: $$q(y,\xi,\mu )=\xi ^{2}+cy^{2\text{ }}-1-i\mu.$$ For
$\mu = 0$, $Q$ has 2 characteristics: $(0,1)$ and
$(0,-1)$. \\
For $\mu $ small enough, $Q$ is elliptic outside fixed neighborhoods
$V_{(0,1)}$ and $V_{(0,-1)}$ respectively of $(0,1)$ and $(0,-1)$.
\subsection {First case: if $(x,\xi) \in V_{(0,1)}$}
 Let $\eta =\xi -1 $, then the symbol of $Q$ is
\[
q(x,\eta ,\mu )  =  \eta ^{2}+2\eta +cx^{2}-i\mu.
\]
We have $$q(0,0,0 )=0,$$ and $$\partial_{\eta}q(0,0,0 )=2.$$ Using
the Weierstrass factorization theorem, we get a factorization of
$q(x,\eta ,\mu )$ of the form
\begin{equation}\label{eq4}
q(x,\eta ,\mu )=a(x,\eta,\mu )(\eta -\Lambda (x,\mu )),
\end{equation} where
$a(x,\eta,\mu )\neq 0$ and
\[
\begin{array}{lll}
\Lambda(x,\mu ) & = & -1+(1-cx^{2}+i\mu )^{\frac{1}{2}}  \\ & = &
-\dfrac{c}{2}x^{2}+i \dfrac{\mu}{2}  +O(\mu ^{2},\mu x^{2},x^{4}).
\end{array}
\]
Using the factorization (\ref{eq4}) and referring to ~\cite{horn60} for
some technical constructions (see sections 6 and 7), we obtain the
decomposition of the pseudo-differential operator $Q$:
\begin{equation}\label{eq5}
Q = A(x,hD,h,\mu )(\frac{h}{i}\partial _{x}-\Lambda (x,\mu )+O_{\mu
}(h)),
\end{equation}
where $O_{\mu }(h)$ is uniformly bounded in $\mu$, $A$ is an
elliptic operator of principal symbol $a(x,\eta ,\mu ) $.\\ First,
we construct a microlocal inverse of the operator
$$\frac{h}{i}\partial _{x}-\Lambda (x,\mu ),$$ and then we prove that
the norm of this inverse is $O(h^{-2/3})$. Such inverse is microlocal since $x$ is assumed to be small enough.\\
Let
\[
(\frac{h}{i} \frac{d}{dx}-\Lambda (x,\mu ))u  = v,
\]
where $u$, $v \in L^{2}([-a_{0},a_{0}])$ for sufficiently small constant $a_{0}>0$.\\
Then $$ (\frac{d}{dx}-\frac{i}{h}\Lambda(x,\mu ) )u =
\frac{i}{h}v.$$ Let $F$ be a primitive of $-\frac{i}{h}\Lambda ,$
then
$$\frac{d}{dx}(e^{F}u)=\frac{i}{h}e^{F}v.$$
The following theorem is needed for the proof.
\begin{theorem}\label{thm2}
 Let $x$
and $y$ verifying $y\geq x$ and $max(|x|,|y|)\leq a_{0}$, then there
exist $\lambda
>0$ and $C>0$ such that $$ -g(x)+g(y)\leq \frac{\lambda
}{h}(x^{3}-y^{3})+C\frac{\mu ^{3/2}}{h},$$ where $g(x) = Re(F(x))$.
\end{theorem}
\begin{proof}
Since $g(x) = Re(F(x))$, then
$$
g^{\prime }(x) =  \frac{1}{2h}(-Im(c)x^{2}+\mu +O(\mu ^{2},\mu
x^{2},x^{4})).
$$
 Let $\mu ^{1/2}t = x$ and$\quad \mu ^{1/2}s =y.$\\
Consider the function $K$ such that $K(t)=g(\mu ^{1/2}t)$, then
 $$K^{\prime }(t)=\frac{\mu ^{3/2}}{2h}[
-Im(c)t^{2}+1+O(\mu ,\mu t^{2},\mu t^{4})].$$Take
$$K(t)=\frac{\mu ^{3/2}}{2h}[ -\frac{Im(c)}{3}t^{3}+t+O(\mu
t,\mu t^{3},\mu t^{5})].
$$
 \subsubsection {If $\|(x,y)\| \leq c_{1}\mu ^{1/2}$ for small positive constant
$c_{1}$ independent of $\mu$.} We have $\| (s,t)\| \leq c_{1}$ and
\begin{equation}\label{eq6}
-K(t)+K(s)\leq \frac{\mu ^{3/2}}{2h}[\frac{Im(c)}{3}%
(t^{3}-s^{3})-(t-s)+\gamma ],
\end{equation}
for small constant $\gamma >0$.
\begin{lem}\label{lem1}
For all $0<\varepsilon <1$ and $s\geq t$, we have
\[
s-t\leq \varepsilon (s^{3}-t^{3})+\frac{1}{\varepsilon}.
\]
\end{lem}
\begin{proof}
It is evident if $s-t < 1 $.\\
If $s-t \geq 1$, then
\begin{equation} \label{eq7} 
	(s-t)^{2} \leq 4(s^{2}+ st
+t^{2})(s-t) = 4(s^{3}-t^{3}).\end{equation}
For any $ a \in
 \mathbb{R} ,$ we have
$$a \leq \frac{1}{2}(\varepsilon a^{2} +\frac{1}{\varepsilon}),$$
then:$$\forall \varepsilon > 0 \ \ , \ s-t \leq
\frac{\varepsilon}{2} (s-t)^{2} +\frac{1}{2 \varepsilon}.$$ From
(\ref{eq7}), we have:$$s-t\leq 2 \varepsilon (s^{3}-t^{3})+\frac{1}{2
\varepsilon },$$ and we achieve the proof of Lemma \ref{lem1}.\end{proof}
We deduce from (\ref{eq6}) and
Lemma \ref{lem1} that
$$
-K(t)+K(s)  \leq  \frac{\mu ^{3/2}}{2h}[(\frac{Im(c)}{3}%
-\varepsilon )(t^{3}-s^{3})+\frac{1}{\varepsilon }+\gamma ],$$ and
that for sufficiently small $\varepsilon$, there exist $\lambda
>0$ and $C>0$ such that \begin{equation}\label{eq8} 
	 -K(t)+K(s) \leq \frac{\mu
^{3/2}}{h}(\lambda (t^{3}-s^{3})+C),
\end{equation} 
for $s$ and $t$ verifying $ s\geq t $ and  $ \| (s,t) \| \leq c_{1}$.\\
Return now to $x$ and $y$ such that $\mu ^{1/2%
}t=x$, $\mu ^{1/2}s =y$ and $y\geq x$. It follows from (\ref{eq8}) that
\begin{equation}\label{eq9} 
-g(x)+g(y)\leq \frac{\lambda }{h}(x^{3}-y^{3})+C\frac{\mu
^{3/2}}{h}.
\end{equation}
\subsubsection {If $\| (x,y) \| \geq c_{2} \mu ^{1/2} $ for some constant
$c_{2} > 0 $.} We have
$$
-g(x)+g(y)  =  \frac{1}{2h} \int_{x}^{y}(-Im(c)t^{2}+\mu +O(\mu
^{2},\mu t^{2},t^{4}))dt, $$ and there exists $c_{0}>0 $ such that
$$ -g(x)+g(y) \leq \frac{1}{2h}[-\frac{Im(c)}{3}(y^{3}-x^{3})+\mu
(y-x)+c_{0}(\mu ^{2}(y-x)+\mu (y^{3}-x^{3})+y^{5}-x^{5})].
$$
We have
\[
\begin{array}{lll}
y^{3}-x^{3} & \geq & \frac{1}{2} (y-x)(x^{2}+y^{2})\\
& & \\
 &\geq & \frac{1}{2}c_{2}^{2}\mu (y-x),
\end{array}
\]
and since $\|(x,y)\|\leq \delta$ (we can take
$\delta=\sqrt{2}a_{0}$), we have
\[
\begin{array}{llll}
y^{5}-x^{5} & =  & (y-x)(y^{4}+y^{3}x+y^{2}x^{2}+yx^{3}+x^{4})\\
& & \\
&\leq & 3(y-x)(x^{2}+y^{2})^{2}\\
& & \\
& \leq & 3\delta^{2}(y-x)(x^{2}+y^{2})\\
& & \\
 & \leq & 6\delta^{2}(y^{3}-x^{3}).
\end{array}
\]
We deduce that
$$ 
-g(x)+g(y) \leq  \frac{1}{2h}[-\frac{Im(c)}{3}(y^{3}-x^{3})+(2 /c_{2}^{2})%
(y^{3}-x^{3})+ 2\mu (c_{0}/c_{2}^{2})(y^{3}-x^{3}) +\mu
c_{0}(y^{3}-x^{3})+6c_{0}\delta^{2}(y^{3}-x^{3})].$$ For
 $c_{2}$ large enough, we conclude that for some constant $\lambda >0$
\[
-g(x)+g(y)\leq -\frac{\lambda}{h}(y^{3}-x^{3}).
\]
The inequality (\ref{eq9}) remains valid in this case and we achieve the
proof of Theorem \ref{thm2}.
\end{proof}
For $0<a\leq a_{0}$, we define with a rate of $O(e^{-\frac{c}{h}})$ the parametrix
 \begin{equation}\label{eq10}  u(x)=\frac{i}{h}\int_{a
}^{x}e^{-F(x)+F(y)}v(y)dy.
\end{equation}
then
\[
u(x)=Kv(x),
\]
where $K$ is an integral operator with the Kernel function
\[
K(x,y)=-\frac{i}{h}e^{-F(x)+F(y)}\chi _{\left( x\leq y\leq
a\right)},
\]
 We use Schur lemma to estimate $\|K\|_{L^{2}([-a_{0},a_{0}])}$.
\begin{lem}\label{lem2}
$$\|K\|_{L^{2}(I)}\leq (\sup_{x\in I}\int_{I} | K(x,y)| dy)^{1/2}(\sup_{y\in I}\int_{I} |
K(x,y)| dx)^{1/2},$$where  $I=[-a_{0},a_{0}]$.\end{lem} 
\begin{proof}
Cauchy-Schwarz inequality implies that
$$|Kv(x)|^{2}\leq\int_{I} |K(x,y)||v(y)|^{2}dy\int_{I}
|K(x,y)|dy.$$ The last integral is estimated by $\sup_{x\in
I}\int_{I} \left| K(x,y)\right| dy$, an integration with respect to
$x$ gives
$$\int_{I}|Kv(x)|^{2}dx \leq(\sup_{x\in I}\int_{I} | K(x,y)| dy)\int_{I}\int_{I}  |K(x,y)||v(y)|^{2}dxdy,$$
then
$$\int_{I}|Kv(x)|^{2}dx \leq(\sup_{x\in I}\int_{I} | K(x,y)| dy)(\sup_{y\in I}\int_{I} |
K(x,y)| dx)\int_{I}|v(y)|^{2}dy,$$ which completes the
proof.\end{proof}
Theorem \ref{thm2} implies that
\[
| K(x,y)| \leq \frac{1}{h} e^{(\lambda /h)(x^{3}-y^{3})} e^{C
\mu ^{3/2}/h} \chi _{( x\leq y\leq a)},
\]
\[
\int_{x}^{a}| K(x,y)| dy\leq \frac{1}{h}e^{C\mu ^{3/
2}/h}\int_{x}^{\infty}e^{-(\lambda/h)(y^{3}-x^{3})}dy,
\]
\[
\int_{-a_{0}}^{y}| K(x,y)| dx\leq \frac{1}{h}e^{C\mu ^{3/
2}/h}\int_{-\infty}^{y}e^{-(\lambda/h)(y^{3}-x
^{3})}dx.
\]
Let: $$h^{1/3}\tilde{x}=x, \ h^{
1/3}\tilde{y}=y.$$ We have
\begin{equation}\label{eq11} 
\sup_{x \in [-a_{0},a_{0}]}\int_{x}^{a}| K(x,y)| dy\leq h^{-2/3}e^{
C\mu ^{3/2}/h} \int_{\tilde{x}}^{+\infty }e^{{-\lambda(\tilde{y}}
^{3}{-\tilde{x}}^{3}{ )}}d\tilde{y},
\end{equation}
and
\begin{equation}\label{eq12} 
\sup_{y \in [-a_{0},a_{0}]}\int_{-a_{0}}^{y}| K(x,y)| dx\leq h^{-2/3}e^{%
C\mu ^{3/2}/h}\int_{-\infty }^{\tilde{y}}e^{{ -\lambda(\tilde{y}}%
^{3}{ -\tilde{x}}^{3}{ )}}d\tilde{x}.
\end{equation}
\begin{lem}\label{lem3}
\[
\int_{x}^{+\infty }e^{-\lambda(y^{3}-x^{3})}dy <\infty \text{ and }
\int_{-\infty }^{y}e^{-\lambda(y^{3}-x^{3})}dx <\infty.
\]
\end{lem}
\begin{proof} Let
\[
I(x)=\int_{x}^{+\infty }e^{{-\lambda(y}^{3}{ -x}^{3}{
)}}dy.
\]
$\ast $ If $x$ is bounded, it is evident since $\int^{+\infty
}_{x}e^{-\lambda y^{3}}dy < \infty$.\newline $\ast $ If $x$ is
unbounded:
\newline Let $y = t+x $, we have
$$
I(x) =  e^{\lambda x^{3}}\int_{0}^{+\infty }e^{-{\lambda(}%
x^{3}+3x^{2}t+3xt^{2}+t^{3})}dt,$$ $$ I(x) =  \int_{0}^{+\infty
}e^{-{ \lambda(}3x^{2}t+3xt^{2}+t^{3})}dt.$$

$\ast \ast $ When $x\rightarrow -\infty $, if $t = -xs$ then
\[
I(x)= - x \int_{0}^{+\infty }e^{\lambda x^{3}(\text{ }%
3s-3s^{2}+s^{3})}ds.
\]
Let $\Psi (s)=3s-3s^{2}+s^{3}$, then  $$\frac{\Psi
(s)}{s}=3-3s+s^{2}\geq \alpha,\text{ for some }\alpha> 0,$$ which
implies that
\[
I(x)\leq - x \int_{0}^{+\infty }e^{{\lambda\alpha }x^{3}s}ds =
O(\frac{1}{\text{ } x ^{2}})< \infty.
\]

$\ast \ast $ When $x\rightarrow +\infty $, if $t = xs$ then
\[
I(x)=x\int_{0}^{+\infty }e^{-\lambda (\text{
}3s+3s^{2}+s^{3})x^{3}}ds.
\]
$$
I(x)  \leq  x\int_{0}^{+\infty }e^{-\lambda x^{3}s}ds.$$$$ I(x) \leq
\frac{1}{\lambda x^{2}}< \infty.$$\vspace{.4cm}\\ Replace $I(x)$ by
the integral $\int_{-\infty }^{y}e^{-\lambda(y^{3}-x^{3})}dx $, we
obtain the same result and we achieve the proof of Lemma \ref{lem3}.
\end{proof}
Finally, by applying the Schur lemma, inequalities (\ref{eq11}), (\ref{eq12})
and Lemma \ref{lem3}, we obtain
\begin{eqnarray}\label{eq13}
\| (\frac{h}{i}\partial _{x}-\Lambda (x,\mu ))^{-1}\| & \leq &
(\sup_{x}\int_{x}^{a}| K(x,y)| dy)^{1/2}(\sup_{y}\int_{-a_{0}}^{y}|
K(x,y)| dx)^{1/2} \nonumber \\
 &\leq & h^{-2/3} e^{C\mu ^{3/2}/h}.
\end{eqnarray}
From (\ref{eq3}), (\ref{eq5}) and \ref{eq13}, we conclude that
\[
\| P^{-1}\| =O(h^{1/3})=O((Rez)^{-1/3}).
\]
Then
\[
\|( H_{c}-zI) ^{-1}\| =O(| z | ^{-1/3}).%
\]
\subsection {Second case: if $(x,\xi)\in V_{(0,-1)}$}
Let $\eta
=$ $\xi +1$ then the symbol of $Q$ is
\[
\begin{array}{lll}
q(x,\eta ,\mu ) & = & \eta ^{2}-2\eta +cx^{2}-i\mu,
\end{array}
\]
and
\[
q(x,\eta ,\mu )=a(x,\eta ,\mu )(\eta -\Lambda (x,\mu )),
\]
where
\[
\Lambda (x,\mu )=\frac{c}{2}x^{2}-i\frac{\mu }{2}+O(\mu ^{2},\mu
x^{2},x^{4}).
\]
It is enough to replace (\ref{eq10}) by the parametrix
\[
(\frac{h}{i}\partial _{x}-\Lambda (x,\mu
))^{-1}v=\frac{i}{h}\int_{-a }^{x}e^{-F(x)+F(y)}v(y)dy.
\]
Now, assume that $Im(c) < 0 $ and $ Imz > 0$.\\
$\ast $ if $(x,\xi) \in V_{(0,1)}$:
\[
\begin{array}{lll}
q(x,\xi ,\mu ) & = &  \xi^{2}+cx^{2}-1-i\mu.
\end{array}
\newline
\]
Let $\eta =\xi-1$, then
\[
\begin{array}{lll}
q(x,\eta ,\mu ) & = & \eta ^{2}+2\eta +cx^{2}-i\mu \\&  &\\ & = &
a(x,\eta ,\mu )(\eta -\Lambda (x,\mu )),
\end{array}
\]
\[
\Lambda (x,\mu )=-\frac{c}{2}x^{2}+i\frac{\mu }{2}+O(\mu ^{2},\mu
x^{2},x^{4}).
\]
Take the parametrix:
\[
u(x)=\frac{i}{h}\int_{-a }^{x}e^{-F(x)+F(y)}v(y)dy.
\]
$\ast$ If $(x,\xi) \in V_{(0,-1)}$:
\[
\Lambda (x,\mu )=\frac{c}{2}x^{2}-i\frac{\mu }{2}+O(\mu ^{2},\mu
x^{2},x^{4}),
\]
\[
u(x)=\frac{i}{h}\int_{a }^{x}e^{-F(x)+F(y)}v(y)dy.
\]
Finally if $ Imz < 0$, take $-\mu $ instead of $\mu
$. \\
Recall that outside $V_{(0,1)}$ and $V_{(0,-1)}$, $Q$
is an elliptic operator and then we have $$\| P^{-1}\| = O(h).$$ We
conclude that $ \| P^{-1}\| =O(h^{1/3})$ is verified in this case
and we achieve the proof of Theorem \ref{thm1}. \end{proof}
\begin{rem}
Boulton shows that
\[
\lim_{\eta \rightarrow +\infty }\|( H_{c}-(b\eta +c\eta ^{p}))
^{-1}\| = +\infty \text{ \ for \ }\frac{1}{3}< p< 3.
\]
By applying Theorem \ref{thm1}, we deduce that for $0 < p\leq
\frac{1}{3}$:
\[
\lim_{\eta \rightarrow +\infty }\| ( H_{c}-(b\eta +c\eta ^{p}))
^{-1}\| = 0.
\]
\end{rem}
\section{Numerical results}
\subsection{Spectral discretization of the complex harmonic oscillator}
The Chebfun system contains algorithms that yield high-precision resolution spectral collocation methods on Chebyshev grids. The Chebops tools in the Chebfun system for solving differential equations are summarized in ~\cite{TrefethenF}, ~\cite {Driscoll}, ~\cite {TrefethenT}, ~\cite{TrefethenS} and ~\cite {Hale}. \\
The Chebops implementation effectively merges the numerical analysis approach of spectral collocation with the computational methodology of spectral discretization matrices. \\
The Spectral Collocation method for solving differential equations consists of constructing weighted interpolants of the form:
\begin{equation}\label{eq14}
	u(x)\approx P_{N}(x)=\sum_{j=0}^{N}\frac{\alpha(x)}{\alpha(x_{j})}\phi_{j}(x)u_{j},
\end{equation}	
where $x_{j}$ for $j=0,....,N$ are interpolation nodes, $\alpha(x)$ is a weight function, $$u_{j}=u(x_{j}),$$ and the interpolating functions
$\phi_{j}(x)$ satisfy $$\phi_{j}(x_{k})=\delta_{j,k}$$ and $$u(x_{k})=P_{N}(x_{k})$$ for $k= 0,....,N$.\\ Hence $P_{N}(x)$ is an interpolant of the function $u(x)$.\\ By taking $l$ derivatives of (\ref{eq14}) and evaluating the result at the nodes  $x_{j}$, we get:
$$u^{(l)}(x_{k})\approx \sum_{j=0}^{N}\frac{d^{l}}{dx^{l}}\left[\frac{\alpha(x)}{\alpha(x_{j})}\phi_{j}(x)\right]_{x=x_{k}}, \quad k= 0,....,N.$$
The entries define the differentiation matrix:  
$$D^{(l)}_{k,j}= \frac{d^{l}}{dx^{l}}\left[\frac{\alpha(x)}{\alpha(x_{j})}\phi_{j}(x)\right]_{x=x_{k}}.$$
The derivatives values $u^{(l)}$ are approximated  at the nodes $x_{k }$ by $ D^{(l)}u $.\\
The derivatives are converted to a differentiation matrix form and the differential equation problem is transformed into a matrix eigenvalue problem.\\
The eigenfunctions 	$u(x)$ of the eigenvalue problem approximate finite terms of Chebyshev polynomials as
\begin{equation} \label{eq15}P_{N}(x)= \sum_{j=0}^{N}\phi_{j}(x)u_{j},
\end{equation}
where the weight function $\alpha(x)$=1, $\phi_{j}(x)$ is the Chebyshev polynomial of degree $\leq N$ and $u_{j}=u(x_{j})$.\\
The Chebyshev collocation points are defined by: 
\begin{equation} \label{eq16} x_{j}=cos(\frac{j\pi}{N}),\quad j=0,....,N. \end{equation}
A spectral differentiation matrix for the Chebyshev collocation points is created by interpolating a polynomial through the collocation points, i.e., the polynomial
$$P_{N}(x_{k})= \sum_{j=0}^{N}\phi_{j}(x_{k})y_{j}.$$
The derivatives values of the interpolating polynomial  (\ref{eq15}) at the Chebyshev collocation points  (\ref{eq16}) are:
$$ P_{N}^{(l)}(x)=\sum_{j=0}^{N}\phi_{j}^{(l)}(x_{k})y_{j}.$$The differentiation matrix $D^{(l)}$ with entries 
\begin{equation}\label{eq17}
	D^{(l)}_{k,j}=\phi_{j}^{(l)}(x_{k})\end{equation}
 is explicitly determined in ~\cite {TrefethenF} and  ~\cite {Weiden}.\\
For further information on convergence rates, collocation differentiation matrices, and the efficiency of the Chebyshev collocation method, refer to ~\cite{TrefethenF, Weiden,CANUTO,Gottlieb}.\\
For the complex
harmonic oscillator, $H_{c}$ has been discretized by a Chebyshev
collocation spectral method on a finite interval $[-L,L]$ with
Dirichlet boundary conditions.
We define the matrix $H^{N}_{c}$ with
 entries $$(H^{N}_{c})_{ij}=-P_{j}^{(2)}(x_{i})+cx^{2}_{i}P_{j}(x_{i}),$$
 for all $1 \leq i
 , j \leq N-1 $. Then the Chebyshev-approximation of $H_{c}$ is defined by 
 $$H^{N}_{c}=-D_{N}^{(2)}+S,$$ where $D_{N}^{(2)}$ is the second order spectral differentiation matrix defined in (\ref{eq17}) rescaled to $[-L,L]$ instead of $[-1,1]$ and $$S=diag(cx^{2}_{1},cx^{2}_{2},....,cx^{2}_{N-1}).$$
\subsection{Numerical Computation of pseudo-spectra}
 A starting point for computations is an equivalent definition of pseudo-spectra. $\sigma_{\varepsilon}(A)$ can be characterized by eigenvalues of perturbed matrices (see ~\cite{ TrefethenF}):
$$\sigma_{\varepsilon}(A) =\{z\in{\mathbb{C}}: z\text{ is an eigenvalue of } A+E \text{ for some } E \text{ with } \lVert E \rVert\leq \varepsilon
\}.$$
If $\rVert.\rVert$ is the 2-norm, as is convenient and physically appropriate in most
applications, then a further equivalence is
\[ \sigma _{\varepsilon }(H_{c}^{N} ) = \{ z \in {\mathbb{C}} : \sigma_{min}(H^{N}_{c}- zI )\leq \varepsilon
\},
\]
where $\sigma_{min}(H^{N}_{c}$) denotes the minimum singular value of
$H_{c}^{N}$.
We calculate $\sigma_{min}(H^{N}_{c}-
zI )$ on a grid of points
$z_{ij}$ using Chebfun algorithms.\\
We show that the pseudo-spectra of the matrix $H^{N}_{c}$ become
much larger when the order of the discretization $N $ increases and
tends to represent the pseudo-spectra of $H_{c}$.\\
Figure \ref{fig:Fig1} shows the boundaries of some $\varepsilon$-pseudo-spectra of the approximate matrix $ H^{N}_{c} $
 for  $ N = 200 $, $L = 6 $ and $ c = 1 + 5i$.\\
The numerical computations confirm the
theoretical results. The first eigenvalues of the matrix $ H^{N}_{c}
$ are regularly spaced numbers along a ray close to the line
$c^{\frac{1}{2}}\mathbb{R}_{+}$. A spectral stability characterizes these eigenvalues. However, far from this line, the resolvent norm appears to grow exponentially as $\lvert z \lvert\rightarrow \infty$ along any ray between
$0$ and $arg(c)$, so every value of $z$ sufficiently far out in this infinite
sector is a pseudo-eigenvalue for an exponentially small value of $\varepsilon$. We conclude the high-energy instability of the spectrum under small perturbations in physical problems described by this operator. Other numerical computations can be found in ~\cite{Daviesn20,TrefethenF}.
\begin{figure}[h]
	\caption{Eigenvalues and $\varepsilon$-pseudo-spectra of the
		matrix $H_{c}^{N}$ for $N = 200$, $c=1+5i$. The left column gives the
		corresponding values of $\varepsilon$.}
\includegraphics[width=16cm]{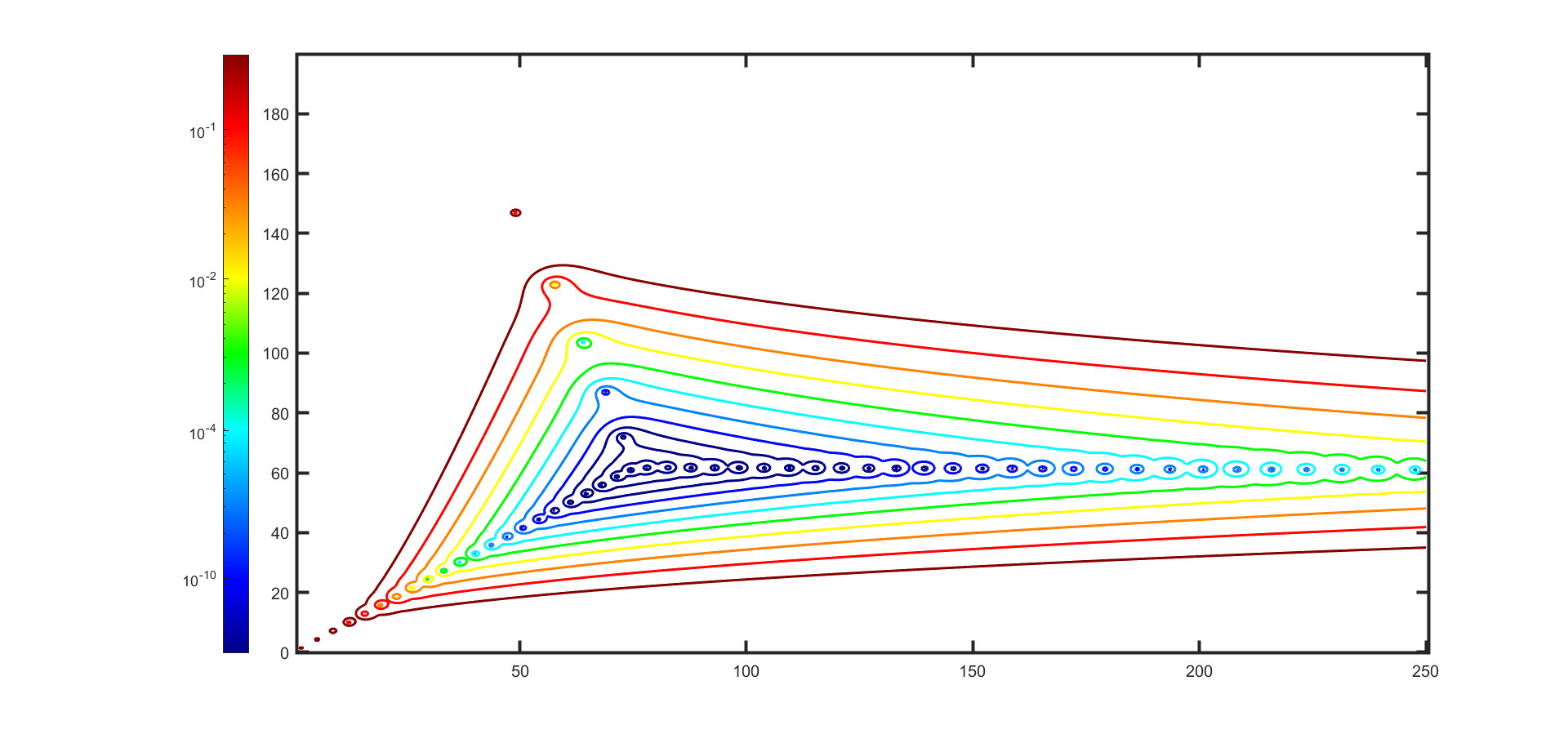}
\label{fig:Fig1}.
\end{figure}
	\section{Conclusion}

	Our study builds on the foundational work of E.B. Davies, L.S. Boulton, and N. Trefethen, who have contributed significantly to our understanding of resolvent norms in the context of complex potentials. By focusing on the resolvent norm of the complex harmonic oscillator along curves defined by $z_{\eta }= b\eta + c\eta ^{p} $ where $ b > 0$,
	$ \frac{1}{3}< p <3 $ independent of $\eta> 0$, we have established a framework for analyzing the spectral properties of these operators. In particular, we have refined our focus to the case where $p = \frac{1}{3}$, providing a more precise analysis that provides a crucial understanding of the resolvent norm behavior and the associated pseudo-modes. The pseudo-eigenvalues provide a clear interpretation of the stability and instability phenomena that can arise in the presence of complex potentials. Stability problems are of growing interest in mathematical physics due to their implications in quantum mechanics and non-Hermitian operator theory. Moreover, exploring the implications of our results in higher dimensions could yield significant information into the pseudospectral properties of multi-dimensional Schrödinger operators with complex potentials. Although the challenges associated with higher-dimensional analysis are non-trivial, significant advancements may result from understanding the spectral properties of more complex systems.

\bibliographystyle{unsrtnat}
\bibliography{references}

\begin{thebibliography}{27}
\providecommand{\natexlab}[1]{#1}
\providecommand{\url}[1]{\texttt{#1}}
\expandafter\ifx\csname urlstyle\endcsname\relax
  \providecommand{\doi}[1]{doi: #1}\else
  \providecommand{\doi}{doi: \begingroup \urlstyle{rm}\Url}\fi

\bibitem[Bender and S.A.(1978)]{Orszag}
C.M. Bender and Orszag S.A.
\newblock Advanced mathematical methods for scientists and engineers.
\newblock \emph{McGraw-Hill, New York}, 1978.

\bibitem[Gana(2023)]{Gana2}
S.~Gana.
\newblock Numerical computation of spectral solutions for sturm-liouville eigenvalue problems.
\newblock \emph{International Journal of Analysis and Applications.}, 2023.
\newblock \doi{10.28924/2291-8639-21-2023-86}.

\bibitem[Mhadhbi et~al.(2024)Mhadhbi, Gana, and Alsaeedi]{Gana3}
N.~Mhadhbi, S.~Gana, and M.~Alsaeedi.
\newblock Exact solutions for nonlinear partial differential equations via a fusion of classical methods and innovative approaches.
\newblock \emph{Scientific Reports}, 14, 2024.
\newblock \doi{https://doi.org/10.1038/s41598-024-57005-1}.

\bibitem[Mhadhbi et~al.(2023{\natexlab{a}})Mhadhbi, Gana, and Alsaeedi]{Gana5}
N.~Mhadhbi, S.~Gana, and M.~Alsaeedi.
\newblock Exact solutions for nonlinear partial differential equations; a fusion of classical methods and innovative approaches.
\newblock \emph{arXiv:2310.14082.2023}, 14, 2023{\natexlab{a}}.
\newblock \doi{https://doi.org/10.48550/arXiv.2310.14082}.

\bibitem[Zworski(2002)]{Zworskin130}
M.~Zworski.
\newblock Numerical linear algebra and solvability of partial differential equations.
\newblock \emph{Communications in Mathematical Physics}, pages 293--307, 2002.
\newblock \doi{10.1007/s00220-002-0683-6}.

\bibitem[Mhadhbi et~al.(2023{\natexlab{b}})Mhadhbi, Gana, and Alharbi]{Mhadhbi1}
N.~Mhadhbi, S.~Gana, and H.~Alharbi.
\newblock Classes of second order nonlinear partial differential equations reducible to first order.
\newblock \emph{arXiv:2305.03128}, 2023{\natexlab{b}}.
\newblock \doi{10.48550/arXiv.2305.03128}.

\bibitem[Mhadhbi et~al.(2023{\natexlab{c}})Mhadhbi, Gana, and Alharbi]{Mhadhbi2}
N.~Mhadhbi, S.~Gana, and H.~Alharbi.
\newblock Exact solutions of classes of second order nonlinear partial differential equations reducible to first order.
\newblock \emph{International Journal of Advanced and Applied Sciences. Volume 10, issue 10, 78-85}, 2023{\natexlab{c}}.
\newblock \doi{10.21833/ijaas.2023.10.009}.

\bibitem[Davies(1999{\natexlab{a}})]{Daviesn20}
E.B. Davies.
\newblock Pseudo-spectra, the harmonic oscillator and complex resonances.
\newblock \emph{Proceedings: Mathematical, Physical and Engineering Sciences}, 455:\penalty0 585--599, 1999{\natexlab{a}}.
\newblock \doi{https://doi.org/10.1098/rspa.1999.0325}.

\bibitem[Davies(2000)]{Daviesn30}
E.B. Davies.
\newblock Pseudospectra of differential operators.
\newblock \emph{Journal of Operator Theory}, 43\penalty0 (2):\penalty0 243--262, 2000.
\newblock ISSN 03794024, 18417744.

\bibitem[Davies(1999{\natexlab{b}})]{Daviesn40}
E.B. Davies.
\newblock Semi-classical states for non-self-adjoint schrödinger operators.
\newblock \emph{Communications in Mathematical Physics}, 200:\penalty0 35--41, 1999{\natexlab{b}}.
\newblock \doi{10.1007/s002200050521}.

\bibitem[Trefethen(1999)]{Trefethenn100}
L.N. Trefethen.
\newblock Computation of pseudospectra.
\newblock \emph{Acta Numerica}, 8:\penalty0 247–295, January 1999.
\newblock \doi{10.1017/s0962492900002932}.

\bibitem[Sjöstrand(2002-2003)]{Johannesn120}
J.~Sjöstrand.
\newblock Pseudospectrum for differential operators.
\newblock \emph{Séminaire Équations aux dérivées partielles}, pages 1--9, 2002-2003.
\newblock URL \url{http://eudml.org/doc/11058}.

\bibitem[Roch and Silbermann(1996)]{Rochn140}
S.~Roch and B.~Silbermann.
\newblock C*-algebra techniques in numerical analysis.
\newblock \emph{Journal of Operator Theory}, 35\penalty0 (2):\penalty0 241--280, 1996.
\newblock ISSN 03794024, 18417744.

\bibitem[Trefethen(1997)]{Trefethenn90}
L.N. Trefethen.
\newblock Pseudospectra of linear operators.
\newblock \emph{SIAM Review}, 39\penalty0 (3):\penalty0 383--406, 1997.
\newblock \doi{10.1137/S0036144595295284}.

\bibitem[Kato(1995)]{Katon50}
T.~Kato.
\newblock \emph{Perturbation theory for linear operators}.
\newblock Springer Berlin, Heidelberg, 1995.
\newblock \doi{10.1007/978-3-642-66282-9}.

\bibitem[Hörmander(1971)]{horn60}
L.~Hörmander.
\newblock Uniqueness theorems and wave front sets for solutions of linear differential equations with analytic coefficients.
\newblock \emph{Communications on Pure and Applied Mathematics}, 24\penalty0 (5):\penalty0 671–704, 1971.
\newblock \doi{10.1002/cpa.3160240505}.

\bibitem[Hörmander(1990)]{horn70}
L.~Hörmander.
\newblock \emph{The analysis of linear partial differential operators}.
\newblock Springer eBooks, 1990.
\newblock URL \url{http://ci.nii.ac.jp/ncid/BA11553887}.

\bibitem[Trefethen(2000)]{TrefethenF}
L.N. Trefethen.
\newblock Spectral methods in matlab.
\newblock \emph{SIAM}, 2000.

\bibitem[Boulton(2002)]{Boultonn10}
L.S. Boulton.
\newblock Non-self-adjoint harmonic oscillator, compact semigroups and pseudospectra.
\newblock \emph{Journal of Operator Theory}, 47:\penalty0 413--429, 2002.
\newblock ISSN 03794024, 18417744.

\bibitem[Dencker et~al.(2003)Dencker, Sjöstrand, and Zworski]{Zworski}
N.~Dencker, J.~Sjöstrand, and M.~Zworski.
\newblock Pseudospectra of semiclassical (pseudo‐) differential operators.
\newblock \emph{Communications on Pure and Applied Mathematics}, 57\penalty0 (3):\penalty0 384–415, 2003.
\newblock \doi{10.1002/cpa.20004}.

\bibitem[Driscoll et~al.(2008)Driscoll, Bornemann, and Trefethen]{TrefethenS}
T.A. Driscoll, F.~Bornemann, and L.N. Trefethen.
\newblock The chebop system for automatic solution of differential equations.
\newblock \emph{BIT Numerical Mathematics}, 2008.

\bibitem[Driscoll et~al.(2014)Driscoll, Hale, and Trefethen]{Driscoll}
T.A. Driscoll, N.~Hale, and L.N. Trefethen.
\newblock Chebfun guide.
\newblock \emph{Pafnuty publications. Oxford}, 2014.

\bibitem[Aurent and Trefethen(2017)]{TrefethenT}
J.L. Aurent and L.N. Trefethen.
\newblock Block operators and spectral discretizations.
\newblock \emph{Siam Review. Vol 59.No 2. 423-446}, 2017.

\bibitem[Driscoll et~al.(2012)Driscoll, Hale, and Trefethen]{Hale}
T.A. Driscoll, N.~Hale, and L.N Trefethen.
\newblock Chebfun-numerical computing with functions.
\newblock \emph{http://www.chebfun.org}, 2012.

\bibitem[Weideman and Reddy(2000)]{Weiden}
J.A. Weideman and S.C. Reddy.
\newblock A matlab differentiation matrix suite.
\newblock \emph{ACM Trans. Math. Softw. 26, 465–519}, 2000.

\bibitem[Canuto et~al.(1988)Canuto, Hussaini, Quarteroni, and ZANG]{CANUTO}
C.~Canuto, M.Y. Hussaini, A.~Quarteroni, and T.A. ZANG.
\newblock Spectral methods in fluid dynamics.
\newblock \emph{Springer-Verlag, Berlin, Germany}, 1988.

\bibitem[Gottlieb and Orszag(1977)]{Gottlieb}
D.~Gottlieb and S.A. Orszag.
\newblock \emph{Numerical Analysis of Spectral Methods}.
\newblock Society for Industrial and Applied Mathematics, 1977.
\newblock \doi{10.1137/1.9781611970425}.

\end{thebibliography}

\end{document}